\documentclass[11pt]{amsart} \usepackage{latexsym, amssymb, stmaryrd}

\newtheorem{thm}{Theorem}[section]
\newtheorem{lemma}[thm]{Lemma}
\newtheorem{prop}[thm]{Proposition}
\newtheorem{cor}[thm]{Corollary}

\theoremstyle{definition}
\newtheorem{df}[thm]{Definition}
\newtheorem{rmk}[thm]{Remark}
\newtheorem{rmks}[thm]{Remarks}
\newtheorem{ex}[thm]{Example}
\newtheorem{fact}[thm]{Fact}
\newtheorem{question}[thm]{Question}

\newcommand{\Z}{\mathbb{Z}}

\def \r { {\mathbb R} }
\def \<{\langle}
\def \>{\rangle}
\def \n {\mathbb N}

\def \*Z {{{^*}\Z}}
\def \g{\mathfrak{g}}
\def \((  {(\!(}
\def \)) {)\!)}

\def \st {\operatorname {st}}

\def \ns{\operatorname{ns}}
\def \pns{\operatorname{pns}}

\newcommand{\hh}{\hat{\mathfrak{h}}}
\newcommand{\hH}{\hat{H}}
\newcommand{\hexp}{\hat{\exp}}
\def \g{\mathfrak{g}}
\def \h{\mathfrak{h}}
\def \p{\hat{p}}
\def \gamg{\Gamma_{\mathfrak{g}}}
\def \gamh{\Gamma_{\hat{\mathfrak{h}}}}
\def \u{U^{[1]}}
\def \f{f^{[1]}}
\def \uu{U^{[k]}}
\def \ff{f^{[k]}}
\def \uuu{U^{[k+1]}}
\def \fff{f^{[k+1]}}
\def \lin{\operatorname{Lin}}
\def \flin{\operatorname{FLin}}
\def \int{\operatorname{int}}
\numberwithin{equation}{section}
\def \l{\llbracket}
\def \rr{\rrbracket}

\allowdisplaybreaks[2]

\begin{document}

\title{Nonstandard Hulls of Locally Exponential Lie Algebras}
\author{Isaac Goldbring} 
\address {University of Illinois, Department of Mathematics, 1409 W. 
Green street, Urbana, IL 61801}

\email{igoldbr2@math.uiuc.edu}
\urladdr{www.math.uiuc.edu/~igoldbr2}
\thanks{This research was supported by an Arnold O. Beckman Research Award.}

\begin{abstract}
We show how to construct the nonstandard hull of certain infinite-dimensional Lie algebras in order to generalize a theorem of Pestov on the enlargeability of Banach-Lie algebras.  In the process, we consider a nonstandard smoothness condition on functions between locally convex spaces to ensure that the induced function between the nonstandard hulls is smooth.  We also discuss some conditions on a function between locally convex spaces which guarantee that its nonstandard extension maps finite points to finite points.
\end{abstract}

\maketitle

\

\section{Introduction}

\

\noindent In the early 1990s, Pestov \cite{P} gave a nonstandard hull construction for Banach-Lie algebras and groups and used it to prove the following theorem on the enlargeability of Banach-Lie algebras.

\begin{thm}\label{T:pestov}
Let $\g$ be a Banach-Lie algebra.  Suppose that there exists a family $\mathcal{H}$ of closed Lie subalgebras of $\g$ and a neighborhood $V$ of $0$ in $\g$ such that:

\begin{itemize}
\item For each $\h_1,\h_2\in \mathcal{H}$, there is an $\h_3\in \mathcal{H}$ such that $\h_1\cup \h_2\subseteq \h_3$;
\item $\bigcup \mathcal{H}$ is dense in $\g$;
\item every $\h\in \mathcal{H}$ is enlargeable and if $H$ is a corresponding connected, simply connected Lie group, then the restriction $\exp_H|V\cap \h$ is injective.
\end{itemize}
Then $\g$ is enlargeable.
\end{thm}

\

\noindent  Recall that a Banach-Lie algebra is \textit{enlargeable} if it is isomorphic to the Lie algebra of a Banach-Lie group.  (Unlike the finite-dimensional situation, there are Banach-Lie algebras which are not enlargeable; see, for example, \cite{LT}.)

\

\noindent  Nowadays, it is recognized that the setting of Banach-Lie groups and algebras is too restrictive when studying infinite-dimensional Lie groups and algebras and the proper model spaces for such groups and algebras are arbitrary \textit{locally convex spaces}.  Let us remind the reader that a locally convex space is a Hausdorff topological vector space $E$ for which there is a basis of neighborhoods of $0$ consisting of convex sets.  Equivalently, a locally convex space is a vector space $E$ equipped with a separating family of seminorms; these seminorms yield a topology on $E$ for which a subbase of open sets around $0$ are sets of the form
$$V(p,\epsilon):=\{x\in E \ | \ p(x)<\epsilon\},$$
as $p$ varies over the family of seminorms and $\epsilon$ ranges over $\r^{>0}$.  For an introduction to infinite-dimensional Lie theory as it is now studied, see the wonderful survey \cite{N} by Karl-Hermann Neeb.  From now on in this paper, when we speak of Lie groups and algebras, we always mean Lie groups and algebras modeled on locally convex spaces.  

\

\noindent Since arbitrary Lie groups and algebras lack much of the structure theory of their finite-dimensional counterparts, one usually adds extra assumptions on the groups/algebras to be able to develop an adequate Lie theory.  A Banach-Lie group is an example of a \textit{locally exponential Lie group}, which is a Lie group possessing a smooth exponential function which provides a diffeomorphism between an open neighborhood of $0$ in its Lie algebra and an open neighborhood of the identity in the Lie group.  There is a corresponding notion of a \textit{locally exponential Lie algebra}, which is a Lie algebra that is a natural candidate to be the Lie algebra of a locally exponential Lie group (see Section 5 for precise definitions of these notions).  It is one of the open problems in the Neeb survey (\cite{N}, Problem VI.6)  to generalize Pestov's theorem to the class of locally exponential Lie algebras.

\

\noindent  Just as in the Banach setting, we can construct the nonstandard hull of an arbitrary internal Lie algebra and prove that it is also a Lie algebra (modeled on the nonstandard hull of the original model space).  Whereas the saturation assumption on the nonstandard extension yields quite easily that the nonstandard hull of an internal Banach-Lie algebra is once again a Banach-Lie algebra, it is not at all immediate that the hull of a locally exponential Lie algebra is also a locally exponential Lie algebra.  This is due to the fact that locally exponential Lie algebras are defined in terms of smooth functions on the underlying locally convex space, and internally smooth functions do not necessarily induce smooth functions on the nonstandard hull.  This is why we have to strengthen the notion of smoothness in our locally exponential Lie algebras to ensure that the nonstandard hull is once again a locally exponential Lie algebra.  

\

\noindent In the Banach setting, if an internal Lie algebra is enlargeable, the theory of the Baker-Campbell-Hausdorff (BCH) series allows one to construct the nonstandard hull of the corresponding Banach-Lie group in a straightforward manner.  We are not as fortunate in our setting, and so our theorem requires an extra (necessary) hypothesis relating the local group operations of the various subalgebras in $\mathcal{H}$ in order to construct the nonstandard hull of an internal Lie group whose Lie algebra is an element of $\mathcal{H}^*$.

\

\noindent We do not assume that the reader is familiar with infinite-dimensional Lie theory and so all relevant notions will be defined. 

\

\noindent We assume that the reader is familiar with elementary nonstandard analysis; otherwise, consult ~\cite{D} or ~\cite{He} for a friendly introduction.  Let us say that all nonstandard arguments take place in a sufficiently saturated nonstandard extension.  

\

\noindent Let us mention a few conventions that we use throughout the paper.  We always suppose $m$ and $n$ range over $\n:=\{0,1,2,\ldots\}$.  For any set $A$, $A^{\times n}$ denotes the cartesian product  $\underbrace{A\times \cdots \times A}_{n \text{ times}}.$  If $G$ is a group and $A\subseteq G$, then $A^n$ denotes the set of $n$-fold products from $A$, i.e. $$A^n:=\{a_1\cdots a_n \ | \ a_i\in A \text{ for all }i=1,\ldots,n\}.$$  For any topological space $X$ and any $a\in X$, we let $$\mu_X(a):=\bigcap \{ \mathcal{O}^* \ | \ \mathcal{O} \text{ is an open neighborhood of }a \text{ in }X\}.$$  If the space $X$ is clear from context, we write $\mu(a)$ instead of $\mu_X(a)$.  We also let $X_{\ns}:=\bigcup_{a\in X} \mu(a)$.  

\

\noindent Now suppose $X$ is a locally convex space and $\Gamma$ is a set of seminorms defining the topology on $X$.  Then for $Y$ an internal subset of $X^*$, we define the set $$Y_f:=\{ x\in Y \ | \ p(x)\in \r_{\ns} \text{ for all }p\in \Gamma\}.$$  If $Y=X^*$, we will just write $X_f$ for this set instead of $X_f^*$.  We will also let $\mu_Y(0)=\{ x\in Y \ | \ p(x)\in \mu_\r(0) \text{ for all }p\in \Gamma\}$ and we sometimes write $\mu(Y)$ or $\mu_Y$ for this set.  As before, if $Y=X^*$, we just write $\mu(X)$ or $\mu_X$, and note that this is equal to $\mu_X(0)$ as defined in the previous paragraph.  Finally, for $a,b\in X^*$, we write $a\sim b$ if $a-b\in \mu(X)$.

\

\noindent I would like to thank Lou van den Dries, Ward Henson, and Karl-Hermann Neeb for very helpful discussions.

\

\section{Nonstandard Hulls of Internal Lie Algebras}

\

\noindent In this section, we work with the following setting.  We let $\g$ be a locally convex Lie algebra, i.e. $\g$ is a locally convex space equipped with a continuous Lie bracket $[\cdot,\cdot]:\g \times \g \to \g$.  We let $\gamg$ denote the set of \textit{all} continuous seminorms on $\g$.  We further suppose that $\h$ is an internal subalgebra of $\g^*$, i.e. $\h$ is an internal linear subspace of $\g$ such that $[\h,\h]\subseteq \h$.  Our goal in this section is to form the nonstandard hull of $\h$.  

\

\begin{lemma}
$\h_f$ is a real Lie algebra and $\mu_\h$ is a Lie ideal of $\h_f$.
\end{lemma}

\begin{proof}
It is well-known and easy to see that $\h_f$ is a real vector space and $\mu_\h$ is a real subspace of $\h_f$.  We first show that $[\h_f,\h_f]\subseteq \h_f$.  Since $[\cdot,\cdot]$ is continuous at $(0,0)$, given $p\in \Gamma_\g$, there exist $q\in \gamg$ and $r\in \r^{>0}$ such that for all $a,b\in \g$, if $q(a),q(b)<r$, then $p([a,b])<1$.  Since $x,y\in \h_f$, we can choose $\alpha\in \r^{>0}$ so that $q(\alpha x),q(\alpha y)<r$.  Then $p([\alpha x,\alpha y])<1$, whence $p([x,y])<\frac{1}{\alpha^2}$.  It thus follows that $[x,y]\in \h_f$ and $\h_f$ is a real Lie algebra.

\

\noindent It remains to show that $[\h_f,\mu_\h]\subseteq \mu_\h$.  Suppose $x\in \h_f$ and $y\in \mu_\h$.  Let $p\in \gamg$ and let $\epsilon \in \r^{>0}$.  By continuity of $[\cdot,\cdot]$ at $(0,0)$, there exists $q \in \gamg$ and $r\in \r^{>0}$ such that for all $a,b\in \g$, if $q(a),q(b)<r$, then $p([a,b])<\epsilon$.  Since $x\in \h_f$, we can choose $\alpha \in \r^{>0}$ so that $q(\alpha x)<r$.  Since $y\in \mu_\h$, we have that $q(\frac{1}{\alpha}y)<r$, whence we can conclude that $p([x,y])=p([\alpha x,\frac{1}{\alpha}y])<\epsilon$.  But $[\alpha x,\frac{1}{\alpha}y]=[x,y]$, whence we see that $p([x,y])<\epsilon$.  Since $p$ and $\epsilon$ were arbitrary, we see that $[x,y]\in \mu_\h$. 
\end{proof}

\

\noindent We can now define the nonstandard hull of $\h$ to be $\hh:=\h_f/\mu_\h$, which is a real Lie algebra.  (The Lie bracket is given by $[x+\mu_\h,y+\mu_\h]:=[x,y]+\mu_\h$.)  Let $\pi_\h:\h_f\to \hh$ denote the canonical projection map.  For $p\in \gamg$, define $\p:\hh\to \r$ by $\p(x+\mu_\h):=\st(p(x))$.  (This is well-defined by the observation that for all $x,y\in \g$, $|p(x)-p(y)|\leq p(x-y)$.)  

\

\noindent Let $\gamh:=\{\p \ | \ p\in \gamg\}.$  We claim that $\gamh$ is a separating family of seminorms on $\hh$.  It is trivial to verify that each $\p$ is a seminorm.  To see that $\gamh$ is separating, note that if $x\in \h_f\setminus \mu_\h$, then for some $p\in \gamg$ and some $\epsilon\in \r^{>0}$, $p(x)\geq \epsilon$.  Consequently, $\p(x+\mu_\h)\geq \epsilon$.

\

\noindent We now see that $\hh$ equipped with the family of seminorms $\gamh$ is a locally convex space.  It remains to show that the Lie bracket of $\hh$ is continuous with respect to the locally convex topology just given to $\hh$.

\

\begin{lemma} $[\cdot,\cdot]:\hh \times \hh \to \hh$ is continuous and thus $\hh$ is a locally convex Lie algebra.
\end{lemma}

\begin{proof}
We first show the continuity of $[\cdot,\cdot]$ at $(0+\mu_\h,0+\mu_\h)$.  Let $\p\in \gamh$ and let $\epsilon \in \r^{>0}$.  Fix $\epsilon'\in \r^{>0}$ with $\epsilon'<\epsilon$.  Choose $q\in \gamg$ and $r\in \r^{>0}$ so that if $a,b\in \g$ and $q(a),q(b)<r$, then $p([a,b])<\epsilon'$.  Then if $\hat{q}(x+\mu_\h),\hat{q}(y+\mu_\h)<r$, one has $\p([x,y]+\mu_\h)\leq \epsilon'<\epsilon$.

\

\noindent We next show that for any $c+\mu_\h\in \hh$, the map $$x+\mu_\h\mapsto [c,x]+\mu_\h:\hh \to \hh$$ is continuous at $0+\mu_\h$ (and hence continuous on all of $\hh$ since the aforementioned map is linear).  Fix $\p\in \gamh$ and $\epsilon\in \r^{>0}$.  Fix $\epsilon'\in \r^{>0}$ with $\epsilon'<\epsilon$.  As in the above paragraph, choose $q\in \gamg$ and $r\in \r^{>0}$ so that if $a,b\in \g$ satisfy $q(a),q(b)<r$, then $p([a,b])<\epsilon'$.  Choose $\alpha \in \r^{>0}$ so that $q(\alpha c)<r$.  Now suppose $\hat{q}(x+\mu_\h)<\alpha r$.  Then $q(\frac{1}{\alpha}x)<r$, whence $p([\alpha c,\frac{1}{\alpha}x])<\epsilon'$.  Thus, $\p([c,x]+\mu_\h)<\epsilon$.

\

\noindent An analogous argument shows that for any $c+\mu_\h \in \hh$, the map $$x+\mu_\h\mapsto [x,c]+\mu_\h:\hh \to \hh$$ is also continuous.  We can thus conclude that $[\cdot,\cdot]:\hh \to \hh$ is continuous from the fact that for a topological vector space $X$, a bilinear map $T:X\times X \to X$ is continuous if it is continuous at $(0_X,0_X)$ and if for each $a\in X$, the functions $x\mapsto T(a,x):X\to X$ and $x\mapsto T(x,a):X\to X$ are continuous.  (This is probably well known, but here is a nonstandard proof of this.  Suppose $(a,b)\in X\times X$ and $(c,d)\in \mu(a,b)$.  Then 
$$T(a,b)-T(c,d)=T(a-c,b)+T(c-a,b-d)+T(a,b-d),$$ which is in $\mu(X)$ by our assumptions.)
\end{proof}

\

\begin{rmk}
It is obvious that the linear map $\iota:\g \to \hat{\g^*}$ given by $\iota(x)=x+\mu_\h$ is such that for every $p\in \gamg$ and every $x\in \g$, one has $p(x)=\p(\iota(x))$.  In particular, $\iota$ is a continuous injection.
\end{rmk}

\begin{rmk}
An easy saturation argument shows that $\hh$ is a closed subspace of $\hat{\g^*}$.  Since $\hat{\g^*}$ is complete (see Theorem 3.15.1 of \cite{L}), it follows that $\hh$ is complete as well.
\end{rmk}

\

\section{Nonstandard Differentiability Conditions in Locally Convex Spaces}

\

\noindent In this section, we define a nonstandard notion of smoothness for functions between locally convex spaces which is stronger than the standard notion of smoothness and show how such functions induce (standardly) smooth functions on the nonstandard hulls.  We then introduce a standard strengthening of smoothness which implies our nonstandard notion.  Finally, we show that for certain locally convex spaces, our nonstandard notion is equivalent to ordinary smoothness.

\

\noindent Throughout this section, we assume $E$ and $F$ are locally convex spaces, $U\subseteq E$ is open, and $f:U\to F$ is a function.  Before we enter our discussion of differentiability, we first provide the following easy lemma, which may be well-known but is included here for the sake of completeness.  Using the terminology of \cite{S}, we let $\lin^k(E^*,F^*)$ denote the space of internal $k$-linear maps from $E^*$ to $F^*$ and we introduce the space $$\flin^k(E^*,F^*)=\{T\in \lin^k(E^*,F^*) \ | \ T((E_f)^k))\subseteq F_f\}.$$  We let $\flin(E^*,F^*)$ denote $\flin^1(E^*,F^*)$.

\

\begin{lemma}\label{L:flin}
Suppose $T\in \lin(E^*,F^*)$.  Then $T\in \flin(E^*,F^*)$ if and only if $T(\mu(E))\subseteq \mu(F).$
\end{lemma}

\begin{proof}
First suppose that $T\in \flin(E^*,F^*)$  yet there is $x\in \mu(E)$ with $T(x)\notin \mu(F)$.  Choose $N\in \n^*\setminus \n$ such that $Nx\in \mu(E)$; such $N$ exists by Theorem 1.6 of \cite{HM}.  But now $T(Nx)=NT(x)\notin F_f$, a contradiction.  

\

\noindent Next suppose that $T(\mu(E))\subseteq \mu(F)$ yet there is $x\in E_f$ such that $T(x)\notin F_f$.  Choose a continuous seminorm $p$ on $F$ and $N\in \n^*\setminus \n$ so that $p(T(x))>N$.  Then $\frac{1}{N}x\in \mu(E)$ while $p(T(\frac{1}{N}x))>1$, contradicting the assumption.
\end{proof}

\

\noindent Let us formulate a higher-order analog of Lemma \ref{L:flin}, whose proof we leave to the reader.

\

\begin{lemma}\label{L:flinhigherorder}
Suppose $T\in \lin^k(E^*,F^*)$.  Then $T\in \flin^k(E^*,F^*)$ if and only if whenever $a_1,\ldots,a_k\in E_f$ and $a_i\in \mu(E)$ for some $i\in \{1,\ldots,k\}$, then $T(a_1,\ldots,a_k)\in \mu(F)$.
\end{lemma}

\

\noindent We now recall the (standard) notion of smoothness that appears in the survey \cite{N} and the stronger (nonstandard) notion defined by Stroyan in \cite{S}.

\

\begin{df}
Let $a\in U$.  Then $f$ is \textbf{differentiable at $a$} if for all $h\in E$, the limit 
$$\lim_{t\to 0} \frac{1}{t}(f(a+th)-f(a))$$ exists.  We denote this limit by $df(a)(h)$ or $D_hf(a)$.  We say that $f$ is \textbf{differentiable} if $f$ is differentiable at $a$ for all $a\in U$.  We say that $f$ is \textbf{$C^1$} if $f$ is differentiable and the map $df:U\times E \to F$ is continuous.  $f$ is said to be $C^k$ if it  is continuous, the iterated directional derivatives $$d^jf(a)(h_1,\ldots,h_j):=(D_{h_j}\cdots D_{h_1}f)(a)$$ exists for all $j\in \{1,\ldots,k\}$, $a\in U$, and $h_1,\ldots, h_j\in E$ and all maps $d^jf:U\times E^j\to F$ are continuous.  Finally, we say that $f$ is \textbf{smooth} if $f$ is $C^k$ for all $k$.
\end{df}

\

\noindent \textbf{Notation:}  If $U$ is an open subset of $E$, we let $$\int(U^*)=\{a\in U^* \ | \ \mu(a)\subseteq U^*\}.$$

\

\begin{df} (Stroyan, \cite{S})
$f$ is \textbf{uniformly differentiable} if there is a map $df:U\to \lin(E,F)$ such that for every $a\in \int(U^*)\cap E_{\ns}$, one has $df(a)\in \flin(E^*,F^*)$, and for every $h\in E_f$ and for every positive $\delta\in \mu(\r)$, we have 
$$\frac{1}{\delta}(f(a+\delta h)-f(a))\sim df(a)(h).$$  The notion $f$ is \textbf{uniformly $C^k$} is defined recursively as follows.  $f$ is uniformly $C^1$ means $f$ is uniformly differentiable.  Suppose $f$ is uniformly $C^k$.  Then we say $f$ is uniformly $C^{k+1}$ if there is a map $d^{k+1}f:U\to \lin^{k+1}(E,F)$ so that whenever $a\in \int(U^*)\cap E_{\ns}$, then $d^{k+1}f(a)\in \flin^{k+1}(E^*,F^*)$ and whenever $x\in E_f$, $h\in (E_f)^k$, and $\delta\in \mu(\r)$ is positive, we have
$$\frac{1}{\delta}(d^kf(a+\delta x)(h)-d^kf(a)(h))\sim d^{k+1}f(a)(h,x).$$  We say that $f$ is \textbf{uniformly smooth} if $f$ is uniformly $C^k$ for every $k$.
\end{df} 

\

\noindent The notion of being uniformly $C^k$ is really a strengthening of the notion of being $C^k$.

\

\begin{lemma}
Suppose $f$ is uniformly $C^k$.  Then $f$ is $C^k$.
\end{lemma}

\

\begin{proof}
For simplicity, we only prove this for the case $k=1$, the higher order cases being similar.  The assumption of uniformly differentiable clearly implies that $f$ is differentiable.  What is left to show is the map $df:U\times E\to F$ is continuous.  Suppose $a\in U$, $a'\in \mu(a)$, $h\in E$, $h'\in \mu(h)$.  We must show $df(a)(h)\sim df(a')(h')$.  By Lemma \ref{L:flin}, we know $df(a)(h)\sim df(a)(h')$.  By Proposition 2.4 of \cite{S}, $df(a)(h')\sim df(a')(h')$.  This completes the proof.
\end{proof}

\

\noindent It seems for our purposes that we will need the following strengthening of Stroyan's definition.

\

\begin{df}
$f$ is \textbf{uniformly differentiable at finite points} if there is a map $df:U\to \lin(E,F)$ such that, for every $a\in \int(U^*)\cap E_f$, one has $df(a)\in \flin(E^*,F^*)$, and for every $h\in E_f$ and for every positive $\delta\in \mu(\r)$, we have
$$\frac{1}{\delta}(f(a+\delta h)-f(a))\sim df(a)(h).$$
\end{df} 

\

\begin{ex} \label{L:stroyanex} (Stroyan, \cite{S})  Let $E=F=\r^\n$ be given its usual structure as a locally convex space, i.e. the topology is generated by the seminorms $p_j$ ($j\in \n$), where, for $a=(a_i)\in E$, $p_j(a):=\max \{|a_1|,\ldots,|a_j|\}$.    Let $f:E\to F$ be the map $f(a)=(\sin (ia_i))$.  Let $a,x\in E_f$ (i.e. $a_i,x_i\in \r_f$ for $i\in \n$).  Then 
\begin{alignat}{2}
(\frac{1}{\delta}((f(a+\delta x)-f(a)))_i&=\sin(ia_i)\frac{\cos(ix_i)-1}{\delta}+\cos(ia_i)\frac{sin(i\delta x_i)}{\delta}\notag \\
						&=ix_i\cos(ia_i)+\delta\cdot z_i \notag \\ \notag
\end{alignat}

\noindent where $z_i\in \r_f$.  By defining $df:E\to \lin(E,F)$ by $df(a)(x)(i):=i\cos(ia_i)x_i$, we see that $f$ is uniformly differentiable at finite points.
\end{ex}

\

\begin{lemma}\label{L:diffcont}
Suppose $f$ is uniformly differentiable at finite points.  Then $f$ is S-continuous at finite points, i.e. if $w\in  \int(U^*)\cap E_f$ and $z\sim w$, then $f(w)\sim f(z)$.
\end{lemma}

\

\begin{proof}
Fix $w$ and $z$ as in the statement of the lemma.  Again by Theorem 1.6 of \cite{HM}, there exists $N\in \n^*\setminus \n$ such that $x:=N(w-z)\in \mu(E)$.  Let $\delta:=\frac{1}{N}$.  By uniform differentiability at $z$, there is $\eta \in \mu(F)$ such that $$\frac{1}{\delta}[f(z+\delta x)-f(z)]=df(z)(x)+\eta,$$ i.e. $$\frac{1}{\delta}[f(w)-f(z)]=df(z)(N(w-z))+\eta.$$  Hence $f(w)-f(z)=df(z)(w-z)+\delta \cdot \eta\in \mu(F)$ by Lemma \ref{L:flin}.
\end{proof}

\

\begin{lemma}\label{L:dfdiffcont}
If $f$ is uniformly differentiable at finite points, then $df$ is S-continuous at finite points, i.e. if $a,a'\in \int(U^*)\cap E_f$ are such that $a\sim a'$, and $x,x'\in E_f$ are such that $x\sim x'$, then $df(a)(x)\sim df(a')(x')$.
\end{lemma}

\begin{proof}
One shows that $df(a)(x)\sim df(a')(x)$ exactly as the proof of Proposition 2.4 in \cite{S}.  Then, since $df(a')\in \flin(E^*,F^*)$, one has $df(a')(x)\sim df(a')(x')$ by Lemma \ref{L:flin}.
\end{proof}

\

\noindent \textbf{Notation:}  In the rest of this paper, for any locally convex space $E$, any internal subspace $Y$ of $E^*$, and any $x\in Y_f$, we may denote the element $x+\mu(Y)$ of $\hat{Y}$ by $\l x \rr$.

\

\noindent For the rest of this subsection, let us assume that $f(U^*\cap E_f)\subseteq F_f$.  (We will take up the issue of when this happens in the next section.)  Since $U$ is open, we can write $$U=\bigcup_{i\in I} \bigcap _{j=1}^{n_i} \{x\in E \ | \ p_{ij}(x-x_{ij})<\epsilon_{ij}\},$$ for some continuous seminorms $p_{ij}$ on $E$, some elements $x_{ij}\in E$ and some $\epsilon_{ij}\in \r^{>0}$.  Let us then define $$\hat{U}:=\bigcup_{i\in I} \bigcap _{j=1}^{n_i} \{ \l x \rr \in \hat{E} \ | \ \hat{p_{ij}}(\l x\rr-\l x_{ij}\rr)<\epsilon_{ij}\}.$$  It is clear that $\hat{U}$ is an open subset of $\hat{E}$ and that if $\l x\rr \in \hat{U}$, then $x\in \int(U^*)\cap E_f$.  If we further assume that $f$ is $S$-continuous at finite points (in particular if $f$ is uniformly differentiable at finite points), then we get a continuous map $\hat{f}:\hat{U}\to \hat{F}$ given by $\hat{f}(\l a\rr)=\l f(a)\rr$.

\

\begin{prop}\label{P:hulldiff}
Suppose $f$ is uniformly differentiable at finite points.  Then $\hat{f}$ is $C^1$.  
\end{prop}

\begin{proof}

\noindent By Lemma \ref{L:dfdiffcont}, we get a map $\hat{df}:\hat{U}\times \hat{E}\to \hat{F}$ given by $\hat{df}(\l a\rr,\l h\rr)=\l df(a)(h)\rr$.  We now show, for $\l a\rr \in \hat{U}$, that $\hat{df}(\l a\rr)$ is the derivative of $\hat{f}$ at $\l a\rr$.  In order to do this, let $\l h\rr\in \hat{E}$, $\hat{p}$ a continuous seminorm on $\hat{E}$ and $\epsilon\in \r^{>0}$.  We need a $\delta\in \r^{>0}$ so that if $|t|<\delta$, then $$\hat{p}(\l df(a)(h)\rr-\frac{1}{t}(\hat{f}(\l a\rr+t\l h\rr)-\hat{f}(\l a\rr)))<\epsilon,$$ i.e. we need a $\delta\in \r^{>0}$ so that if $|t|<\delta$, then 
$$st(p(df(a)(h)-\frac{1}{t}(f(a+th)-f(a))))<\epsilon.$$  Since the above expression is $0$ if $t$ is infinitesimal (by uniform differentiability at finite points), we can find the desired $\delta$ by saturation.

\

\noindent It remains to show $\hat{df}:\hat{U}\times \hat{E}\to \hat{F}$ is continuous.  Fix $[a]\in \hat{U}$ and $[h]\in \hat{E}$.  Let $\hat{p}$ be a continuous seminorm on $\hat{F}$ and $\epsilon\in \r^{>0}$.  We need $r\in \r^{>0}$ and continuous seminorms $\hat{p_1},\ldots,\hat{p_n}$ on $\hat{E}$ so that if $\hat{p_i}([a]-[a']),\hat{p_i}([h]-[h'])<r$ for $i=1,\ldots,n$, then $\hat{p}(\hat{df}([a])([h])-\hat{df}([a'])([h']))<\epsilon$.  If not, then one can use saturation to get $a'\in U^*$ with $a'\sim a$ and $h'\in E_f$ with $h'\sim h$ such that $p(df(a)(h)-df(a')(h'))\geq \epsilon$, which contradicts the S-continuity of $df$ at finite points.
\end{proof}

\

\begin{df}
The notion $f$ is \textbf{uniformly $C^k$ at finite points} is defined recursively as follows.  $f$ is uniformly $C^1$ at finite points means $f$ is uniformly differentiable at finite points.  Suppose $f$ is uniformly $C^k$ at finite points.  Then we say $f$ is uniformly $C^{k+1}$ at finite points if there is a map $d^{k+1}f:U\to \lin^{k+1}(E,F)$ so that whenever $a\in \int(U^*)\cap E_f$, we have $d^{k+1}f(a)\in \flin^{k+1}(E^*,F^*)$ and whenever $x\in E_f$, $h\in (E_f)^k$, and $\delta\in \mu(\r)$ is positive, we have
$$\frac{1}{\delta}(d^kf(a+\delta x)(h)-d^kf(a)(h))\sim d^{k+1}f(a)(h,x).$$  We will say that $f$ is \textbf{uniformly smooth at finite points} if $f$ is uniformly $C^k$ at finite points for every $k\geq 1$.
\end{df}

\

\begin{prop}\label{L:hullsmooth}
If $f$ is uniformly $C^k$ at finite points, then $\hat{f}$ is $C^k$ and $$d^{k}\hat{f}(\l a\rr)(\l h_1\rr,\ldots,\l h_k\rr)=\l d^kf(a)(h_1,\ldots,h_k)\rr.$$  In particular, if $f$ is uniformly smooth at finite points, then $\hat{f}$ is smooth.
\end{prop}

\

\begin{proof}
We prove this by induction on $k$.  The case $k=1$ is exactly Proposition \ref{P:hulldiff} (and its proof).  We now suppose that $f$ is uniformly $C^{k+1}$.  Fix $\l a\rr\in \hat{U}$ and $\l h_1\rr,\ldots,\l h_{k+1}\rr\in \hat{E}$.  We must show that $$d^{k+1}\hat{f}(\l a\rr)(\l h_1\rr,\ldots,\l h_k\rr)=\l d^{k+1}f(a)(h_1,\ldots,h_{k+1})\rr.$$  We first must show that the above expression is well-defined.  So suppose $a'\sim a$ and $h_i'\sim h_i$ for $i=1,\ldots,k+1$.  By the analog of \cite{S}, Proposition 3.2, we know that $d^{k+1}f(a)(h_1,\ldots,h_{k+1})\sim d^{k+1}(a')(h_1,\ldots,h_{k+1})$.  But since $df(a')\in \flin^{k+1}(E^*,F^*)$, Lemma \ref{L:flinhigherorder} shows that $$df(a')(h_1,\ldots,h_{k+1})\sim df(a')(h_1',\ldots,h_{k+1}').$$

\

\noindent  For ease of notation, let $h=(h_1,\ldots,h_k)$ and $d^k_h\hat{f}(\cdot):=d^k\hat{f}(\cdot)(h_1,\ldots,h_k)$.  We now must show that $$\lim_{t\to 0} \frac{1}{t}(d^k_h\hat{f}(\l a\rr+t\l h_{k+1}]\rr-d^k_h\hat{f}(\l a\rr))=\l d^{k+1}f(a)(h_1,\ldots,h_{k+1})\rr.$$  By induction, this amounts to showing that 
$$\lim_{t\to 0} \frac{1}{t}\l d^kf(a+th_{k+1})(h)-d^kf(a)(h)\rr=\l d^{k+1}f(a)(h_1,\ldots,h_{k+1})\rr.$$
Let $\hat{p}$ be a continuous seminorm on $\hat{F}$ and let $\epsilon\in \r^{>0}$.  We need a $\delta\in\r^{>0}$ so that if $|t|<\delta$, then $$\st(p(\frac{1}{t}(d^kf(a+th_{k+1})-d^kf(a)(h))-d^{k+1}f(a)(h_1,\ldots,h_{k+1})))<\epsilon.$$  Since the above quantity is $0$ for infinitesimal $t$, the desired $\delta$ can be obtained by saturation.

\

\noindent Our final obligation is to show that $d^{k+1}\hat{f}:\hat{U}\times \hat{E}^{k+1}\to \hat{F}$ is continuous.  The proof is identical to the corresponding part of the proof of Proposition \ref{P:hulldiff}.
\end{proof}

\

\noindent \textbf{Strong Smoothness}

\

\noindent We now introduce a standard condition on $f$ which implies that it is uniformly differentiable at finite points.  We first need to mention some facts from the calculus of locally convex spaces.  Suppose $f$ is $C^1$.  Let $$\u:=\{(x,y,t)\in U\times E \times \r \ | \ x+ty\in U\},$$ an open subset of $E\times E\times \r$.  Let $\f:\u\to F$ be defined by 

\[
	\f(x,y,t)=
	\begin{cases}
	\frac{1}{t}(f(x+ty)-f(x))	&\text{if $t\not=0$}\\
	df(x)(y) 	&\text{if $t=0$}
	\end{cases}
\]

\noindent It follows from the Mean Value Theorem that $\f$ is continuous.  In fact, it is shown in \cite{BGN} that if $f$ is continuous and there exists a continuous function $\f:\u\to F$ such that $\f(x,y,t)=\frac{1}{t}(f(x+ty)-f(x))$ for $t\not=0$, then $f$ is $C^1$ and $df(x)(y)=\f(x,y,0)$.  

\

\begin{df}
Suppose $f$ is continuous.  Let us say that $f$ is \textbf{strongly $C^1$} if $\f$ is uniformly continuous.
\end{df}

\

\begin{lemma}\label{L:strongfinite}
Suppose $f$ is strongly $C^1$.  Then $f$ is uniformly differentiable at finite points.  
\end{lemma}

\

\begin{proof}
Suppose $a\in \int(U^*)\cap E_f$.  We first show $df(a)\in \flin(E^*,F^*)$.  It suffices to show that if $x\in \mu(E)$, then $df(a)(x)\in \mu(F)$.  But $$df(a)(x)= \f(a,x,0)\sim \f(a,0,0) =df(a)(0)=0$$ since $df(a)$ is an internal linear map.

\

\noindent Now suppose $x\in E_f$ and $\delta$ is a positive element of $\mu(\r)$.  We must show that $\f(a,x,\delta)\sim df(x)(y)$.  But $\f(a,x,\delta)\sim \f(a,x,0)=df(a)(x)$, finishing the proof.
\end{proof}

\

\begin{rmk} 
Notice that we didn't ever use the fact that $a$ and $x$ were finite in the above proof, so being strongly $C^1$ implies uniform differentiability at all points and where we are allowed to take derivatives in the direction of any element of $E^*$.
\end{rmk}

\

\begin{lemma}\label{L:strongunif}
If $f$ is strongly $C^1$, then $f$ is uniformly continuous.
\end{lemma}

\begin{proof}
Suppose $x,y\in U^*$ are such that $x\sim y$.  Choose $N\in \n^*\setminus \n$ such that $z:=N(x-y)\in \mu(E)$.  Let $\delta:=\frac{1}{N}$.  Then $$f(x)-f(y)=f(y+\delta z)-f(y)\sim \delta df(y)(z)=df(y)(\delta z)\in\mu(F).$$ 
\end{proof}

\

\noindent We now will describe the higher order analogs of this notion.  One can recursively define the sets $\uu$ for $k\in \n$ by $\uuu:=(\uu)^{[1]}$ and the functions $\ff:\uu \to F$ by $\fff:=(\ff)^{[1]}$.  It is shown in \cite{BGN} that a a $C^k$ function $f$ is $C^{k+1}$ if and only if $f$ is $C^1$ and $\ff$ is $C^1$.  

\

\begin{df}
The notion $f$ is \textbf{strongly $C^k$} is defined recursively as follows.  The notion $f$ is strongly $C^1$ has already been defined.  Assume $f$ is strongly $C^k$.  Then $f$ is strongly $C^{k+1}$ if $\ff$ is strongly $C^1$.  We will say that $f$ is \textbf{strongly smooth} if $f$ is strongly $C^k$ for all $k\geq 1$.
\end{df}

\

\begin{lemma}\label{L:strongunifk}
If $f$ is strongly $C^k$, then $f$ is uniformly $C^k$ at finite points.
\end{lemma}

\begin{proof}
The proof goes by induction on $k$.  The case $k=1$ is precisely Lemma \ref{L:strongfinite}.  We now assume the result holds for $k$ and we suppose $f$ is strongly $C^{k+1}$.  The induction hypothesis gives us that $f$ is uniformly $C^k$ at finite points.  In order to prove the other two conditions for $f$ to be uniformly $C^{k+1}$ at finite points, we need to elaborate on the relationship between the functions $d^nf$ and $f^{[n]}$ for arbitrary $n$.

\

\noindent Using the terminology from \cite{BGN}, each $d^nf$ is a \textit{partial map} of $f^{[n]}$ in the sense that each $d^nf$ is obtained from $f^{[n]}$ by fixing some coordinates of the domain of $f^{[n]}$.  For example, $$df(x)(h)=\f(x,h,0)$$ and $$d^2f(x)(h_1,h_2)=f^{[2]}(x,h_1,0,h_2,0,0,0).$$  Hence, if $f^{[n]}$ is uniformly continuous, then so is $d^nf$.

\

\noindent Let us show that for any $a\in U^*\cap E_f$ such that $\mu(a)\subseteq U^*$, we have $df(a)\in \flin^{k+1}(E,F).$  Suppose $h_1,\ldots,h_{k+1}\in E_f$ and, without loss of generality, that $h_1\in \mu(E)$.  Then by uniform continuity of $d^{k+1}f$, we have $d^{k+1}f(a)(h_1,\ldots,h_{k+1})\sim d^{k+1}f(a)(0,h_2.\ldots,h_{k+1})=0$.  Now suppose that $h_1,\ldots,h_k,x\in E_f$ and $\delta$ is a positive element of $\mu(\r)$.  We now show that $$\frac{1}{\delta}(d^kf(a+\delta x)(h_1,\ldots,h_k)-d^kf(a)(h_1,\ldots,h_k))\sim d^{k+1}f(a)(h_1,\ldots,h_k,x).$$  This follows from the uniform continuity of $f^{[k+1]}$.  Let us illustrate this in the case when $k=2$, as the formula relating $d^kf$ and $\ff$ is simple enough in this case.  For simplicity, let us denote the left hand side of the above equation by LHS.

\begin{alignat}{2}
\text{LHS} &=\frac{1}{\delta}(f^{[2]}(a+\delta x,h_1,0,h_2,0,0,0)-f^{[2]}(a,h_1,0,h_2,0,0,0)) \notag \\
         &=f^{[3]}((a,h_1,0,h_2,0,0,0),(x,0,0,0,0,0,0),\delta) \notag \\
         &\sim f^{[3]}((a,h_1,0,h_2,0,0,0),(x,0,0,0,0,0,0),0) \notag \\
         &=d^3f(a)(h_1,h_2,x) \notag \qedhere
\end{alignat}
\end{proof}

\

\noindent \textbf{The Case of Complete (HM)-spaces}

\

\noindent Recall that $x\in E^*$ is called \textit{pre-nearstandard} if for every neighborhood $V$ of $0$ in $E$, there is $y\in E$ such that $x-y\in V^*$.  Let $E_{\pns}$ denote the set of pre-nearstandard points of $E$ and note that we always have the inclusions $E_{\ns}\subseteq E_{\pns}\subseteq E_f$.  The importance of the pre-nearstandard points of $E$ is that their image in $\hat{E}$ is the completion of $E$ in $\hat{E}$ (so $E$ is complete if and only if $E_{\ns}=E_{\pns}$).  An \textbf{(HM)-space} is a locally convex space $E$ for which $E_{\pns}=E_f$.

\

\begin{rmks}

\
\begin{enumerate}
\item In standard language, a locally convex space $E$ is an (HM)-space if and only if whenever $\mathcal{F}$ is an ultrafilter on $E$ with the property that for every $U$ from a fixed neighborhood base of $0$ in $E$ there is $n$ such that $nU\in \mathcal{F}$, then $\mathcal{F}$ is a Cauchy filter.
\item For metrizable $E$, $E$ is an (HM)-space if and only if every bounded set is totally bounded. 
\item Examples of (HM)-spaces include the finite-dimensional spaces, (FM)-spaces, nuclear spaces, and Schwarz spaces.
\item $E$ is a \textit{complete} (HM)-space if and only if $E_{\ns}=E_f$.  The space $\r^\n$ from Example \ref{L:stroyanex} is a complete (HM)-space.
\end{enumerate}
\end{rmks}

\noindent The proofs of the above remarks can be found in \cite{HM} and \cite{HM2}.

\
 
\begin{lemma}\label{L:HMsmooth}
Suppose that $E$ is a complete (HM)-space and $f$ is smooth.  Then $f$ is uniformly smooth at finite points.
\end{lemma}

\begin{proof}
We will only show that $f$ is uniformly differentiable at finite points; the argument is the same for higher derivatives.  Suppose $a\in \int(U^*)\cap E_{\ns}$ and $x\in E_{\ns}$.  Then since $df$ is continuous, we know that $df(a)(x)\sim df(\st(a))(\st(x))$, whence $df(a)(x)\in F_{\ns}\subseteq F_f$.  Now suppose that $\delta$ is a positive element of $\mu(\r)$.  Then $$\f(a,x,\delta)\sim \f(\st(a),\st(x),0)=df(\st(a),\st(x))\sim df(a)(x),$$ since $\f$ and $df$ are continuous. 
\end{proof}

\

\section{Finite Functions}

\

\noindent Throughout this section, $E$ and $F$ continue to denote locally convex spaces, but now $U$ denotes an open neighborhood of $0$ in $E$.  We still assume that $f:U\to F$ is any function.

\

\begin{df}
We say that $f$ is a \textbf{finite function} if $f(U^*\cap E_f)\subseteq F_f$.
\end{df}

\

\noindent  In order for $f$ to induce a function on the nonstandard hulls, a necessary requirement is that $f$ be a finite function.  Using Nelson's algorithm (see \cite{Ne}), one can give a standard translation of the notion that $f$ is a finite function, but this ends up being a very complicated condition.  Instead, we seek to prove that $f$ is finite under some natural assumptions.  

\

\noindent Recall that a subset $B$ of a topological vector space $E$ is \textit{bounded} if for any neighborhood $U$ of $0$ in $E$, there exists $n$ such that $B\subseteq nU$.  It is a well-known fact (see Theorem 2.1 of \cite{HM}) that $B$ is bounded if and only if $B^*\subseteq E_f$.  We thus get the following easy lemma.

\

\begin{lemma}
If $f(U)$ is a bounded subset of $F$, then $f$ is a finite function.
\end{lemma}

\

\noindent A less trivial observation is the following.

\

\begin{lemma}\label{L:unifpns}
Suppose $f:U\to F$ is uniformly continuous.  Let $U_1$ be a symmetric open neighborhood of $0$ in $E$ such that $U_1+U_1\subseteq U$.  Then $f(U_1^*\cap E_{\pns})\subseteq F_f$.  In particular, if $E$ is an (HM)-space, then $f|U_1$ is a finite function.  
\end{lemma}

\begin{proof}
Let $x\in U_1^*\cap E_{\pns}$.  We wish to show that $f(x)\in F_f$.  Let $q$ be a continuous seminorm on $F$.  Since $f$ is uniformly continuous, there is a symmetric open neighborhood $V$ of $0$ such that whenever $a,b\in U$ are such that $a-b\in V$, then $q(f(a)-f(b))<1$.  Since $x\in E_{\pns}$, we can find $y\in E$ such that $x-y\in U_1^*\cap V^*$.  Then $y=x+(y-x)\in U^*$, whence $q(f(x)-f(y))<1$.  Since $y$ is standard, $q(f(y))\in \r_f$, whence $q(f(x))\in \r_f$.  Since $q$ was an arbitrary continuous seminorm on $F$, this shows that $f(x)\in F_f$.
\end{proof}

\

\noindent  We can improve Lemma \ref{L:unifpns} if we further assume that $U$ is convex, which is certainly the case for our applications.  Recall that $f$ is said to be \textit{Lipschitz on large distances} if for any continuous seminorm $r$ on $E$ and for any continuous seminorm $q$ on $F$, there is a continuous seminorm $p$ on $E$ so that $q(f(x_1)-f(x_2))\leq p(x_1-x_2)$ for all $x_1,x_2\in U$ for which $r(x_1-x_2)\geq 1$.  We will need the following fact.

\

\begin{fact}  (\cite{CK})  A uniformly continuous mapping from a convex subset of a locally convex space $E$ into a locally convex space $F$ is Lipschitz on large distances.  
\end{fact}

\

\begin{lemma}\label{L:uniffin}
Suppose $f:U\to F$ is uniformly continuous and $U$ is convex.  Then $f$ is a finite function.
\end{lemma}

\

\begin{proof}
Let $x\in U^*\cap E_f$ and let $q$ be a continuous seminorm on $F$.  We wish to show that $q(f(x))\in \r_f$.  Clearly if $x\in\mu(E)$, then by continuity at $0$, we have $f(x)\in \mu(F)$.  We thus may assume that $x\notin \mu(E)$.  Choose a continuous seminorm $r$ on $E$ and $\epsilon \in \r^{>0}$ so that $r(x)\geq \epsilon$.  By replacing $r$ by $\frac{1}{\epsilon}r$, we may assume that $r(x)\geq 1$.  Let $p$ be a continuous seminorm on $E$ such that $q(f(x_1)-f(x_2))\leq p(x_1-x_2)$ for all $x_1,x_2\in U$ for which $r(x_1-x_2)\geq 1$.  Then since $r(x)\geq 1$, we have $q(f(x)-f(0))\leq p(x)\in \r_f$, whence $q(f(x))\in \r_f$.  
\end{proof}

\

\noindent A stronger assumption to impose on $f$ is that it is \textit{Lipschitz}.  Recall that $f$ is said to be Lipschitz if for every continuous seminorm $q$ on $F$, there is a continuous seminorm $p$ on $E$ so that $q(f(x_1)-f(x_2))\leq p(x_1-x_2)$ for every $x_1,x_2\in U$.

\

\begin{lemma}
If $f$ is Lipschitz, then $f$ is a finite function.
\end{lemma}

\begin{proof}
Let $x\in U^*\cap E_f$ and let $q$ be a continuous seminorm on $F$.  Choose $p$ as in the definition of Lipschitz.  Then $q(f(x)-f(0))\leq p(x)\in \r_f$, which implies that $q(f(x))\in \r_f$.
\end{proof}

\

\noindent  We end this section with a question.  For $x_0\in E$ and $p$ a continuous seminorm on $E$, let $B^p_1(x_0)$ denote the set $\{x\in E \ | \ p(x-x_0)<1\}$.  Say that $f$ is \textit{locally Lipschitz} if for every $x_0\in U$ and every continuous seminorm $q$ on $F$, there is a continuous seminorm $p$ on $E$ such that $B^p_1(x_0)\subseteq U$ and $q(f(x)-f(y))\leq p(x-y)$ for all $x,y\in B^p_1(x_0)$.  One has the following result:

\

\begin{fact}  (\cite{G}, Lemma 1.9)  If $f$ is $C^1$, then $f$ is locally Lipschitz.
\end{fact}

\

\begin{question} It does not appear that assuming that $f$ is locally Lipschitz implies that $f$ is a finite function.  Does the assumption that $f$ is uniformly $C^1$ at finite points (or even uniformly smooth at finite points) imply that $f$ is a finite function?
\end{question}

\

\section{Localizing Enlargeability}

\

\noindent In this section, we present our main theorem on localizing enlargeability and some of its corollaries.  We first introduce some of the necessary definitions from locally convex Lie theory.

\

\begin{df}
A \textbf{local Lie group} is a tuple $(G,D,m_G,1)$ such that $G$ is a smooth manifold modeled on a locally convex space, $D\subseteq G\times G$ is open, $m_G:D\to G$ (the product map) is smooth, and such that the following conditions hold:

\begin{itemize}
\item Suppose $xy$ and $yz$ are defined, i.e. $(x,y)\in D$ and $(y,z)\in D$.  Then if one of the products $(xy)z$ or $x(yz)$ are defined, then so is the other and both products are equal;
\item For each $x\in G$, we have $(x,1)\in D$ and $(1,x)\in D$ and $m_G(x,1)=m_G(1,x)=x$;
\item For each $x\in G$, there is a unique $x^{-1}\in G$ such that $m_G(x,x^{-1})=m_G(x^{-1},x)=1$;
\item The map $x\mapsto x^{-1}:G \to G$ is smooth;
\item If $(x,y)\in D$, then $(y^{-1},x^{-1})\in D$.
\end{itemize}
\end{df}

\

\begin{df}\label{D:locexp}
A Lie algebra $\g$ is called \textbf{locally exponential} if there exists a circular, convex open $0$-neighborhood $U\subseteq \g$ and an open subset $D\subseteq \g \times \g$ on which we have a smooth map $$m_U:D\to U, \qquad (x,y)\mapsto x * y$$ such that $(U,D,m_U,0)$ is a local Lie group satisfying:
\begin{enumerate}
\item[(E1)] For $x\in U$ and $|t|,|s|,|t+s|\leq 1$, we have $(tx,sx)\in D$ and $tx * sx=(t+s)x$;
\item[(E2)] The second order term in the Taylor expansion of $m_U$ at $(0,0)$ is $[x,y]$.
\end{enumerate}
\end{df}

\

\noindent The condition (E2) is only there to ensure that the Lie algebra of the local Lie group $(U,D,m_U,0)$ is isomorphic to $\g$.

\

\begin{df}
Let $G$ be a Lie group with Lie algebra $\g$.
\begin{enumerate}
\item[(i)] A \textbf{smooth exponential map} for $G$ is a smooth function $\exp_G:\g\to G$ for which the curves $\gamma_x(t):=\exp_G(tx)$ ($x\in \g$)  are 1-parameter subgroups of $G$ satisfying $\gamma_x'(0)=x$.  (Recall that a 1-parameter subgroup of $G$ is a continuous homomorphism $\r\to G$.)  It is a fact that $G$ can possess at most one smooth exponential function.
\item[(ii)] $G$ is said to be \textbf{locally exponential} if there exists a smooth exponential map $\exp_G$ for $G$, an open $0$-neighborhood $U\subseteq \g$, and an open $e$-neighborhood $V\subseteq G$ such that $\exp_G|U$ is a diffeomorphism of $U$ onto $V$.
\end{enumerate}
\end{df}

\

\noindent The above terminology is due to the fact that the Lie algebra of a locally exponential Lie group is a locally exponential Lie algebra; this is clear from using exponential coordinates.  Thus a locally exponential Lie algebra is a natural candidate to be the Lie algebra of a locally exponential Lie group.

\

\begin{df}
A locally exponential Lie algebra $\g$ is said to be \textbf{enlargeable} if it is the Lie algebra of a locally exponential Lie group.
\end{df}

\

\noindent The outline of the proof of Theorem \ref{T:pestov} is as follows.  An internal Lie algebra $\h\in \mathcal{H}^*$ has been chosen so that $\g$ embeds isometrically into $\hh$ as a closed subalgebra.  By assumption, there is an internal Banach-Lie group $H$ for which $\h$ is its Lie algebra.  As mentioned in the Introduction, $\hh$ is a Banach-Lie algebra and Pestov shows that $\hh$ is enlargeable by constructing the nonstandard hull of $H$, which has $\hh$ as its Lie algebra.  This then finishes the proof as a closed Lie subalgebra of an enlargeable Banach-Lie algebra is enlargeable. 

\

\noindent We now explain the set-up that allows us to pursue the above method of proof for the class of locally exponential Lie algebras.  Suppose that $\g$ is a locally exponential Lie algebra witnessed by $(U,D,m_U,0)$.  For $x,y\in U$, we sometimes write $x*y$ instead of $m_U(x,y)$.  Suppose $\mathcal{H}$ is a directed family of closed subalgebras of $\g$ and suppose that there exists an open, symmetric neighborhood $V\subseteq U$ of $0$ in $\g$ with $V\times V \subseteq D$ and such that 

\begin{enumerate}
\item $\bigcup \mathcal{H}$ is dense in $\g$;
\item for each $\h\in \mathcal{H}$, there is a locally exponential Lie group $H$ such that $L(H)\cong \h$;
\item for each $\h\in \mathcal{H}$, if $H$ is a connected locally exponential Lie group such that $L(H)\cong \h$, then $\exp_H|V\cap \h:V\cap \h \to H$ is injective.
\end{enumerate}

\

\noindent The preceding conditions are the direct analogues of Pestov's assumptions in the Banach setting.  In order to make some of Pestov's arguments go through, it seems necessary to add two further conditions.  In order to explain the new conditions that we assume, let us introduce some notation.  For each $\h \in \mathcal{H}$, let us fix a connected Lie group $H$ such that $L(H)\cong \h$ and let $W_\h$ be an open symmetric neighborhood of $e$ in $H$ contained in $\exp_H(V)$.  Let $D_\h:=\{(x,y)\in \h \times \h \ | \ \exp_H(x)\cdot \exp_H(y)\subseteq W_\h\}$, an open subset of $\h \times \h$.  Define $*_\h:D_\h \to V\cap \h$ by $x*_\h y=\exp_H^{-1}(\exp_H(x)\cdot \exp_H(y))$.  Our new assumptions are that $W_\h$ can be chosen so that there exists a continuous seminorm $p$ on $\g$ for which:

\

\begin{enumerate}
\item[(4)] $\{x\in \h | \ p(x)<1\}^{\times 2} \subseteq D_\h$;
\item[(5)] $\{x\in \g \ | \ p(x)<1\}^{\times 2} \subseteq D$ and $m_U|\{x\in \g \ | \ p(x)<1\}^{\times 2}$ is uniformly continuous.
\end{enumerate}

\noindent  We will need the following consequence of assumption (4).

\

\begin{fact} (\cite{GN}, Proposition 3.7.17)
Assumption (4) implies that for all $\h \in \mathcal{H}$ and all $x,y\in \h$, if $p(x),p(y)<1$, then $x*_\h y=m_U(x,y)$.  (The statement of Proposition 3.7.17 is less precise than what we claim and one needs only to read the proof of that proposition to see that it yields this fact immediately.)
\end{fact}

\

\noindent For the rest of this section, we fix $\h\in \mathcal{H}^*$ and suppose $H$ is a corresponding internal connected locally exponential Lie group whose Lie algebra is $\h$.  For simplicity, let $\exp$ denote the exponential map for $H$.  By the above fact, we can write $x*y$ instead of $x*_\h y$ when $x,y\in \h$ are such that $p(x)<1$ and $p(y)<1$.  

\

\noindent The first step in constructing the nonstandard hull of $H$ is to define $H_f$.  In the Banach setting of Pestov's paper, he defines $H_f$ in a certain fashion, but then ends up showing that $H_f=\bigcup_n \exp(W)^n$, where $W$ is any ball around $0$ in $\h$ of finite, noninfinitesimal radius.  It then follows that $H_f=\bigcup_n \exp(\h_f)^n$.  Indeed, one inclusion is clear, for $W\subseteq \h_f$.  Now suppose $x\in \h_f$.  Choose $m$ so that $\frac{1}{m}x\in W$.  Then $\exp(x)=\exp(\frac{1}{m}x)^m\in \exp(W)^m$, proving the other direction.  I thus propose the following definition in our locally convex setting. 

\

\begin{df}
$H_f:=\bigcup_n \exp(\h_f)^n$.
\end{df}

\

\begin{lemma} $H_f$ is a group.
\end{lemma}

\begin{proof}
$H_f$ clearly contains $e$ and is closed under products.  That $H_f$ is closed under inverses follows from the fact that $\h_f$ is closed under additive inverses and the fact that $\exp(-x)=\exp(x)^{-1}$.
\end{proof}

\

\noindent In analogy with Pestov, we make the following definition.

\

\begin{df}
$\mu_H:=\exp(\mu_\h)$.
\end{df}

\

\noindent The following lemma appears in Pestov's paper, where he uses facts about the BCH series in Banach-Lie algebras to reach this conclusion.  We could not use such an argument and this is where assumptions (4) and (5) make their first appearance.

\

\begin{lemma}
$\mu_H$ is a normal subgroup of $H_f$.
\end{lemma}

\begin{proof}
First suppose that $x,y\in \mu_H$.  Choose $x_1,y_1\in \mu_\h$ such that $\exp(x_1)=x$ and $\exp(y_1)=y$.  Using the continuity of $m_U$ and the fact that $(x,y)\in D_\h$, we have $x_1*y_1\in \mu_\h$ and thus $\exp(x_1*y_1)\in \mu_H$.  But then $$xy=\exp(x_1)\exp(y_1)=\exp(x_1*y_1)\in \mu_H.$$  Since $-x_1\in \mu_\h$, we have $x^{-1}=\exp(-x_1)\in \mu_H$.  Hence, $\mu_H$ is a subgroup of $H_f$.

\

\noindent Now suppose $y\in H_f$ and $x\in \mu_H$.  We will show that $yxy^{-1}\in \mu_H$.  Write $x=\exp(x_1)$ with $x_1\in \mu_\h$ and write $y=y_1\cdots y_n$, with each $y_i=\exp(z_i)$ where $z_i\in \h_f$.  Then $$yxy^{-1}=\exp(z_1)\cdots \exp(z_n)\exp(x_1)\exp(-z_n)\cdots \exp(-z_1).$$  So by induction, it suffices to prove that if $z\in \h_f$ and $w\in \mu_\h$, then $\exp(z)\exp(w)\exp(-z)\in \mu_H$.  Choose a continuous seminorm $q\geq p$ on $\g$ so that for all $x,y\in \g$, if $q(x),q(y)<1$, then $p(x*y)<1$.  Choose $m$ so that $q(\frac{1}{m}z)<1$.  Then by uniform continuity of $m_U$, we know that $\frac{1}{m}z*w\sim \frac{1}{m}{z}$.  Since $p(\frac{1}{m}z*w)<1$, another application of uniform continuity implies that $(\frac{1}{m}z*w)*(-\frac{1}{m}z)\sim (\frac{1}{m}z)*(-\frac{1}{m}z)=0$.  Hence $\exp(\frac{1}{m}z)\exp(w)\exp(-\frac{1}{m}z)\in \mu_H$.  Since $$\exp(z)\exp(w)\exp(-z)=\exp(\frac{1}{m}z)^m\exp(w)\exp(-\frac{1}{m}z)^m,$$ we are finished with the proof of the lemma.
\end{proof}   

\

\begin{df}  We set $\hH:=H_f/\mu_H$ and let $\pi_H:H_f\to \hH$ be the canonical projection map.  
\end{df}

\

\begin{lemma}
Suppose $\pi_\h(x)=\pi_\h(y)$.  Then $\pi_H(\exp x)=\pi_H(\exp y)$.
\end{lemma}

\

\begin{proof}
Choose $m$ so that $p(\frac{1}{m}x),p(\frac{1}{m}y)<1$.  By uniform continuity of $m_U$, $\frac{1}{m}x*(-\frac{1}{m}y)\sim \frac{1}{m}y*(-\frac{1}{m}y)=0$.  Thus $\exp(\frac{1}{m}x*-\frac{1}{m}y)\in \mu_H$, whence $\exp(\frac{1}{m}x)\exp(-\frac{1}{m}y)\in \mu_H$.  Hence 
\begin{alignat}{2}
\pi_H(\exp(x)\exp(y)^{-1})&=\pi_H((\exp \frac{1}{m}x)^{m-1}(\exp \frac{1}{m}x)(\exp (-\frac{1}{m}y))(\exp (-\frac{1}{m}y)^{m-1}) \notag \\
                                             &=\pi_H((\exp \frac{1}{m}x)^{m-1}(\exp (-\frac{1}{m}y))^{m-1}) \notag \\ \notag
 \end{alignat}
Continuing in this fashion, one gets the desired result.
 \end{proof}
 
 \
 
 \noindent The above lemma allows us to define a function $\hexp:\hh\to \hH$ by $$\hexp(x+\mu_\h):=(\exp x)\mu_H.$$ From now on, we will use the notation $\l x \rr$ for $x+\mu_\h$ as introduced earlier in the paper.
 
 \
 
\begin{lemma}\label{L:inj}
$\hexp$ is injective on $\hat{W}:=\{\l x \rr \ | \ \p(\l x\rr)<1\}$.
\end{lemma}

\begin{proof}
Suppose that there exist $\l x_1\rr,\l x_2\rr \in \hat{W}$ such that $\l x_1\rr \not= \l x_2\rr$ but $\hexp(\l x_1\rr)=\hexp(\l x_2\rr)$, i.e. $(\exp x_1)(\exp x_2)^{-1}\in \mu_H$.  Then $\exp(x_1*(-x_2))\in \mu_H$, i.e. there is $z\in \mu_\h$ such that $\exp(x_1*(-x_2))=\exp z$.  Since $\exp$ is injective on $V$, we have $x_1*(-x_2)=z$.  But then, by uniform continuity of $m_U$, $x_2\sim (x_1*(-x_2))*x_2=x_1$, a contradiction.
\end{proof}

\

\begin{rmk}\label{R:S-cont}
All that was used in the above construction of  $\hH$ was that $m_U$ was $S$-continuous on pairs of points of $\g$ with $p$-norm less than $1$ which were \textit{finite}, so the assumption that $m_U$ is uniformly differentiable at finite points allows those arguments to go through.
\end{rmk}

\

\noindent In addition to the conditions (1)-(5) we have imposed on the locally exponential Lie algebra $\g$ and the family of closed subalgebras $\mathcal{H}$, we further assume the condition

\

\begin{itemize}
\item[(6)] $m_U$ is uniformly smooth at finite points. 
\end{itemize}

\

\noindent Note that by uniform smoothness at finite points, Remark \ref{R:S-cont}, and Lemma \ref{L:uniffin}, condition (6) allows us to replace condition (5) with the condition 

\

\begin{itemize}
\item[(5$'$)] $m_U$ is a finite function.
\end{itemize}

\

\noindent Under these conditions, we get a smooth map $\hat{m_U}:\hat{D}\to \hat{\g^*}$.  

\

\begin{lemma}
$\hh$ is a locally exponential Lie algebra.
\end{lemma}

\begin{proof}
Let $\hat{W}$ be as in Lemma \ref{L:inj}.  Let $m_{\hh}:\hat{W} \times \hat{W} \to \hh$ be defined by $m_{\hh}(\l x\rr,\l y\rr)=\l x*y\rr$.  Note that $m_{\hh}=\hat{m_U}|\hat{W} \times \hat{W}$.  Since $\hh$ is a complete subalgebra of $\hat{\g^*}$, it follows that $m_{\hh}$ is smooth.  Let $D':=m_{\hh}^{-1}(\hat{W})$.  We claim that $(\hat{W},D',m_{\hh}|D',\l 0\rr)$ witnesses that $\hh$ is a locally exponential Lie algebra.

\

\noindent It is clear that the above data gives a local group, and since multiplication and inversion are smooth (inversion is in fact linear), we have that $(\hat{W},D',m_{\hh}|D',\l 0\rr)$ is a local Lie group.  We now must verify conditions (E1) and (E2) of Definition \ref{D:locexp}.  Towards proving (E1), suppose $\l x\rr\in \hat{W}$ and $|t|,|s|,|t+s|\leq 1$.  We need $(t\l x\rr,s\l x\rr)\in D'$ and $m_{\hh}(t\l x\rr,s\l x\rr)=(t+s)\l x\rr$.  Now since $\l x\rr\in \hat{W}$, we know that $p(x)<1$, whence $(tx,sx)\in D_\h$ and $tx*sx=(t+s)x$.   It thus follows that $m_{\hh}(t\l x\rr,s\l x\rr)=(t+s)\l x\rr\in \hat{W}$ and so $(t\l x\rr,s\l x\rr)\in D'$.  

\

\noindent Now suppose $h=(h_1,h_2)\in \g\times \g$.  Since $d^2m_U(0,0)(h,h)=[h_1,h_2]$, Proposition \ref{L:hullsmooth} implies that $$d^2\hat{m_U}(\l 0\rr,\l 0\rr)(\l h\rr,\l h\rr)=\l [h_1,h_2]\rr=[\l h_1\rr,\l h_2\rr].$$  It thus follows that the Lie algebra of the local group $(\hat{W},D',m_{\hh}|D',\l 0\rr)$ is $\hh$.  Hence (E2) holds and the proof is finished. 
\end{proof}

\

\begin{lemma}
$\hh$ is enlargeable.
\end{lemma}

\begin{proof}
By Lemma \ref{L:inj}, $\hat{\exp}|\hat{W}$ is injective.  It is also clear from the definitions that $\hat{\exp}$ is a local group morphism when $\hh$ is endowed with the local group structure from the previous lemma.  Let $\hat{H_1}$ denote the subgroup of $\hat{H}$ generated by $\exp(\hat{W})$.  Then Corollary II.2.2 of \cite{N} implies that $\hat{H_1}$ carries the unique structure of a Lie group so that $\hat{\exp}|\hat{W}$ is a diffeomorphism onto an open subset of $\hat{H_1}$.  Then $\hat{H_1}$ is a locally exponential Lie group with Lie algebra $\hh$, finishing the proof.
\end{proof}

\

\noindent We are now ready to state our main theorem on localizing enlargeability.

\

\begin{thm}\label{T:enlarg}
Suppose $\g$ is a locally exponential Lie algebra and $\mathcal{H}$ is a family of closed subalgebras such that $\g$ and $\mathcal{H}$ satisfy conditions (1)-(6) (or (5') instead of (5)).  Then $\g$ is enlargeable. 
\end{thm}

\begin{proof}
As in \cite{P}, we get an internal $\h\in \mathcal{H}^*$ such that the map $\iota:\g \to \hat{\g^*}$ actually takes values in $\hh$.  For the sake of completeness, let us repeat how this argument goes.  Let $X:=\bigcup \mathcal{H}$.  Consider the following family of internal conditions on $A\in \mathcal{P}_F(X)^*$, the set of hyperfinite subsets of $X^*$:

$$C(g,p,n):=\{A\in \mathcal{P}_F(X)^* \ | \ \text{ there exists } g'\in A \text{ such that } p(g-g')<\frac{1}{n}\},$$ where $g$ ranges over $\g$ and $p$ ranges over a set of continuous seminorms of $\g$ generating the topology on $\g$.  Assumption (1) implies that for each $g\in \g$ we have $\mu(g)\cap X^*\not= \emptyset$, whence the family of internal sets $C(g,p,n)$ has the finite intersection property.  Hence, by saturation, there is an $A\in \mathcal{P}_F(X)^*$ in each $C(g,p,n)$, i.e. $\mu(g)\cap A\not=\emptyset$ for every $g\in \g$.  Since the family $\mathcal{H}$ is directed, there is $\h \in \mathcal{H}^*$ such that $A\subseteq \h$.  This is the desired $\h$.

\

\noindent Since $\iota: \g \to \hh$ is a continuous injection, we can infer that $\g$ is enlargeable from the enlargeability of $\hh$ using \cite{N}, Corollary IV.4.10.
\end{proof}

\

\noindent Let us mention a corollary of this theorem involving only standard notions.  Say that a locally exponential Lie algebra $\g$ is \textbf{strong} if there is a local group $(U,D,m_U,0)$ witnessing that $\g$ is a locally exponential Lie algebra for which $m_U$ is strongly smooth.

\

\begin{cor}\label{T:strong}
If $\g$ is a strong locally exponential Lie algebra with a family $\mathcal{H}$ of closed subalgebra satisfying conditions (1)-(4), then $\g$ is enlargeable.
\end{cor}

\begin{proof}
This follows from the previous theorem, using Lemmas \ref{L:strongunif}, \ref{L:strongunifk}, and \ref{L:uniffin}.
\end{proof}

\

\noindent We can remove some of the assumptions on $\g$ in Theorem \ref{T:enlarg} if $\g$ is a complete (HM)-space.

\

\begin{cor}\label{T:HM}
If $\g$ is a locally exponential Lie algebra modeled on a complete (HM)-space with a family $\mathcal{H}$ of closed subalgebras satisfying conditions (1)-(4), then $\g$ is enlargeable.
\end{cor}

\begin{proof}
It is clear from the continuity of $m_U$ and the fact that the finite points are all nearstandard that $m_U$ is a finite map.  By Lemma \ref{L:HMsmooth}, we know that $m_U$ is uniformly smooth at finite points.
\end{proof}

\

\begin{question}  Corollaries \ref{T:strong} and \ref{T:HM} would be exact analogs of Pestov's Theorem for certain classes of locally exponential Lie algebras if condition (4) were not needed. Can one get rid of assumption (4) in any of the above results?
\end{question}

\

\noindent Pestov draws the following corollary to his theorem.

\

\begin{cor}\label{C:locfd}(Pestov)
If $\g$ is a Banach-Lie algebra which contains a dense subalgebra in which every finitely generated subalgebra is finite-dimensional (or such that every finitely generated subalgebra is solvable), then $\g$ is enlargeable.
\end{cor}

\

\begin{question}
Even if one were able to obtain perfect analogs of Pestov's theorem for arbitrary locally exponential Lie algebras, would one be able to draw similar conclusions as in Corollary \ref{C:locfd}?
\end{question}

\


\begin{thebibliography}{1}

\bibitem{BGN} W. Bertram, H. Gl\"ockner, K.H. Neeb, \textit{Differential Calculus, Manifolds, and Lie Groups over Arbitrary Infinite Fields}, arxiv:math/0303300v1.

\bibitem{CK} H. Corson and V. Klee, \textit{Topological Classification of Convex Sets}, Proc. Symp. Pure Math. VII (Convexity). Providence, R.I.: AMS 1963, 37-51.

\bibitem{D} M. Davis, \textit{Applied Nonstandard Analysis}, John Wiley and Sons Inc., 1977.

\bibitem{G} H. Gl\"ockner, \textit{H\"older Continuous Homomorphisms between Infinite-Dimensional Lie Groups are Smooth}, J. Funct. Anal., Vol. 228, Issue 2 (Nov. 2005)

\bibitem{GN} H. Gl\"ockner, K.H. Neeb, \textit{Infinite-dimensional Lie groups, Vol. I, Basic Theory and Main Examples}, book in preparation.

\bibitem{He} C.W. Henson, \textit{Foundations of Nonstandard Analysis: A Gentle Introduction to Nonstandard Extensions}; Nonstandard Analysis: Theory and Applications, L. O. Arkeryd, N. J. Cutland, and C. W. Henson, eds., NATO Science Series C:, Springer, 2001.

\bibitem{HM} C.W. Henson and L.C. Moore, Jr., \textit{The Nonstandard Theory of Topological Vector Spaces}, Transactions of the American Mathematical Society, Vol. 172. (Oct. 1972), pp. 405-435.

\bibitem{HM2} C.W. Henson and L.C. Moore, Jr., \textit{Invariance of the Nonstandard Hulls of Locally Convex Spaces}, Duke Math. J. 40 (1973), 193-205.

\bibitem{LT}  M. Lazard, J. Tits, \textit{Domaines d'injectivit\'e de l'application exponentielle}, Topology 4 (1966), pp. 315-322.

\bibitem{L} W.A.J. Luxemburg, \textit{A General Theory of Monads}, Internat. Sympos. Applications of Model Theory to Algebra, Analysis, and Probability (Pasadena, Calif., 1967), Holt, Rinehart and Winston, New York, 1969, pp. 18-86.  MR 39 \#6244.

\bibitem{N} K. H. Neeb, \textit{Towards a Lie Theory of Locally Convex Groups}, Japan. J. Math. 1 (2006),  pp. 291-468.  

\bibitem{Ne} E. Nelson, \textit{Internal Set Theory:  A New Approach to Nonstandard Analysis}, Bull. Amer. Math. Soc. 83 (1977), pp. 1165-1198.

\bibitem{P} V. Pestov, \textit{Nonstandard Hulls of Banach-Lie Groups and Algebras}, Nova Journal of Algebra and Geom. 1 (1992), pp. 371--384.

\bibitem{S} K.D. Stroyan, \textit{Infinitesimal Calculus on Locally Convex Spaces: 1. Fundamentals}, Transactions of the American Mathematical Society, Vol. 240. (Jun., 1978), pp. 363-383.

\end{thebibliography}
\end{document}